\documentclass[12pt]{amsart}
\pdfoutput=1
\usepackage{fullpage}
\usepackage{graphicx}
\usepackage{url}

\newcommand{\SL}{\operatorname{SL}}
\newcommand{\lcm}{\operatorname{lcm}}
\newcommand{\ZZ}{\mathbb{Z}}
\newcommand{\QQ}{\mathbb{Q}}
\newcommand{\oh}{\mathcal{O}}
\newcommand{\gp}{\mathfrak{p}}

\newtheorem{thm}{Theorem}

\newtheorem{lem}[thm]{Lemma}
\newtheorem*{conj}{Conjecture}
\theoremstyle{definition}

\theoremstyle{remark}

\newtheorem*{conv}{Conventions}

\begin{document}
\title{Distinguishing Hecke eigenforms}
\author[A. Ghitza]{Alexandru Ghitza}
\thanks{\today}
\address{
  Department of Mathematics and Statistics\\
  The University of Melbourne\\
  Parkville, VIC, 3010\\
  Australia
}
\email{aghitza@alum.mit.edu}
\subjclass[2010]{Primary: 11F11.  Secondary: 11F25, 11F30, 11-04}
\keywords{Modular forms, Hecke eigenforms, Fourier coefficients}

\maketitle

\section{Introduction}\label{sect:intro}
Fourier expansions are a popular way of representing
modular forms: they are explicit and easy to manipulate,
and their coefficients often have arithmetic or combinatorial 
meaning and can therefore be interesting in their own right.
For many applications, an essential question is: How many coefficients
are sufficient in order to determine a modular form?  This was answered first in
the context of congruences modulo a prime by Sturm~\cite{Sturm}.  Let
$f$ and $g$ be modular forms of the same weight $k$ and level
$\Gamma\subset\SL_2(\ZZ)$.  Let $\oh$ be the ring of integers of
the number field containing the coefficients of $f$ and $g$, and let
$\gp$ be a prime ideal of $\oh$.  Sturm's result is that if 
$f\not\equiv g\pmod{\gp}$, then
\begin{equation*}
  \text{there exists }\quad n\leq\frac{k}{12}\,[\SL_2(\ZZ):\Gamma]
  \quad\text{ such that }\quad a_n(f)\not\equiv a_n(g)\pmod{\gp}.
\end{equation*}

This result was extended by Ram Murty~\cite{Murty} in several directions:
working with forms of distinct levels and weights, replacing the ``finite
primes'' with the ``infinite prime'', and considering the important
special case of Hecke newforms.  Our interest lies with Theorem 4
in~\cite{Murty}: let $f$, resp. $g$ be newforms of distincts weights
on $\Gamma_0(N_1)$, resp. $\Gamma_0(N_2)$.  Let $N=\lcm(N_1, N_2)$, then
\begin{equation*}
  \text{there exists }\quad n\leq 4\log^2(N)\quad\text{ such that }
  \quad a_n(f)\neq a_n(g).
\end{equation*}
Murty gives a very elegant proof of this result.  Unfortunately, the key
estimate in the proof only holds for $N$ large enough --- it will follow from
the proof of Lemma~\ref{lem:theta} in the next section that the estimate 
fails for $N$ in the set
\begin{equation*}
  \{1, \ldots, 4, 6, \ldots, 12, 30, \ldots, 33, 210, \ldots, 244\}.
\end{equation*}
The statement itself is trivially false for $N=1,2$.

In this paper, we follow Murty's approach to prove the following statement:

\begin{thm}\label{thm:main}
  Let $f$ and $g$ be cuspidal eigenforms of weights $k_1\neq k_2$ 
  on the group $\Gamma_0(N)$.  Then
  \begin{equation}\label{eqn:main}
    \text{there exists }\quad n\leq 4(\log(N)+1)^2
    \quad\text{ such that }\quad a_n(f)\neq a_n(g).
  \end{equation}
\end{thm}

We also indicate some better asymptotic bounds in~\eqref{eqn:main} 
that follow from this method of proof, and we conclude with a section 
describing a computational experiment that tests how sharp the bounds 
are in the case of forms of level $1$.

\begin{conv}
  We assume that any cuspidal eigenform $f$ has been normalised so that
  the coefficient $a_1(f)=1$.  We use the standard notation $p_k$
  for the $k$-th smallest prime number.
\end{conv}

\section{The proof of Theorem~\ref{thm:main}}\label{sect:proof}
The starting point is the following result extracted from the proof
of Theorem 4 in~\cite{Murty}:

\begin{lem}[Ram Murty]\label{lem:murty}
  Let $f$ and $g$ be cuspidal eigenforms of weights $k_1\neq k_2$ on the 
  group $\Gamma_0(N)$, and let $p$ be the smallest prime not dividing $N$.  
  Then
  \begin{equation*}
    \text{there exists }\quad n\leq p^2\quad\text{ such that }
    \quad a_n(f)\neq a_n(g).
  \end{equation*}
\end{lem}
\begin{proof}
  We proceed by contradiction: suppose $a_n(f)=a_n(g)$ for all $n\leq p^2$.
  We have the 
  well-known recurrence relation for Hecke operators on $\Gamma_0(N)$
  and in weight $k$:
  \begin{equation*}
    T_{p^2} = T_p^2 - p^{k-1}\langle p\rangle,
  \end{equation*}
  where $\langle p\rangle$ is the trivial character on $(\ZZ/N\ZZ)^\times$,
  extended by zero to all of $\ZZ/N\ZZ$.  Since eigenforms are normalised
  so $a_1=1$, the relation between Fourier coefficients and Hecke
  eigenvalues gives
  \begin{eqnarray*}
    a_{p^2}(f) &=& a_p^2(f) - p^{k_1-1}\\
    a_{p^2}(g) &=& a_p^2(g) - p^{k_2-1}.
  \end{eqnarray*}
  By assumption we have $a_{p^2}(f)=a_{p^2}(g)$ and $a_p(f)=a_p(g)$, from
  which we derive $k_1=k_2$, contradicting the hypothesis of the Lemma.
\end{proof}

\begin{proof}[Proof of Theorem~\ref{thm:main}]
  According to Lemma~\ref{lem:murty}, it is enough to show that for any 
  $N\geq 1$ there exists a prime $p \leq 2\log(N)+2$ that does not divide $N$.
  Once again we proceed by contradiction: suppose $N$ is divisible by
  all primes up to $2\log(N)+2$.  Then
  \begin{equation*}
    N\geq\prod_{p\leq 2\log(N)+2} p.
  \end{equation*} 
  Using Chebyshev's function
  \begin{equation*}
    \theta(x)=\sum_{p\leq x}\log(p),
  \end{equation*}
  we can rewrite the previous inequality as
  \begin{equation*}
    \log(N)\geq \sum_{p\leq 2\log(N)+2} \log(p)=\theta(2\log(N)+2).
  \end{equation*}
  It will follow from Lemma~\ref{lem:theta} that the right hand side
  of this inequality is $>\log(N)$ for all $N\geq 1$, which leads to
  a contradiction.
\end{proof}

\begin{lem}\label{lem:theta}
  Chebyshev's function satisfies
  \begin{equation*}
    \theta(2x+2)>x\quad\text{for all }x\geq 0.
  \end{equation*}
\end{lem}

This estimate is an application of Theorem 1.4 in~\cite{Dusart}:

\begin{thm}[Dusart]
  Chebyshev's function satisfies
  \begin{equation*}
    |\theta(x)-x|<3.965\frac{x}{\log^2(x)}\quad\text{for all }x>1.
  \end{equation*}
\end{thm}

It is worth noting that Dusart's results are based on detailed knowledge
of the positions of the first $1.5\times 10^9$ zeros of the Riemann zeta 
function, obtained numerically by Brent, van de Lune, te Riele, and 
Winter~\cite{Brent}.

\begin{proof}[Proof of Lemma~\ref{lem:theta}]
  We start by showing that
  \begin{equation*}
    \theta(2x)>x\quad\text{for all }x>8.356.
  \end{equation*}
  Indeed, suppose $x>8.356$, then we have
  \begin{equation*}
    x>\frac{1}{2}\exp(\sqrt{2\cdot 3.965}) \quad\Leftrightarrow\quad
    \log^2(2x)>2\cdot 3.965.
  \end{equation*}
  By Dusart's estimate, we get
  \begin{equation*}
    \theta(2x)-2x>-3.965\frac{2x}{\log^2(2x)}>-x.
  \end{equation*}

  Since $\theta$ is a step function, it is easy to check
  which values of $x\in [0, 8.356]$ do not satisfy the inequality
  $\theta(2x)>x$, namely
  \begin{equation*}
    x\in [0,3/2)\cup [\log 6, 5/2)\cup [\log 30, 7/2)\cup [\log 210, 11/2).
  \end{equation*}
  The discrepancy is largest at the right ends of the intervals,
  as can be seen in the Figure.  We also notice that 
  translating $\theta(2x)$ by $-1$ along the $x$-axis will
  disentangle the two graphs, in other words
  \begin{equation*}
    \theta(2x+2)>x\quad\text{for all }x\geq 0,
  \end{equation*}
  as claimed.
  \begin{figure*}[h]
    \includegraphics[width=10cm]{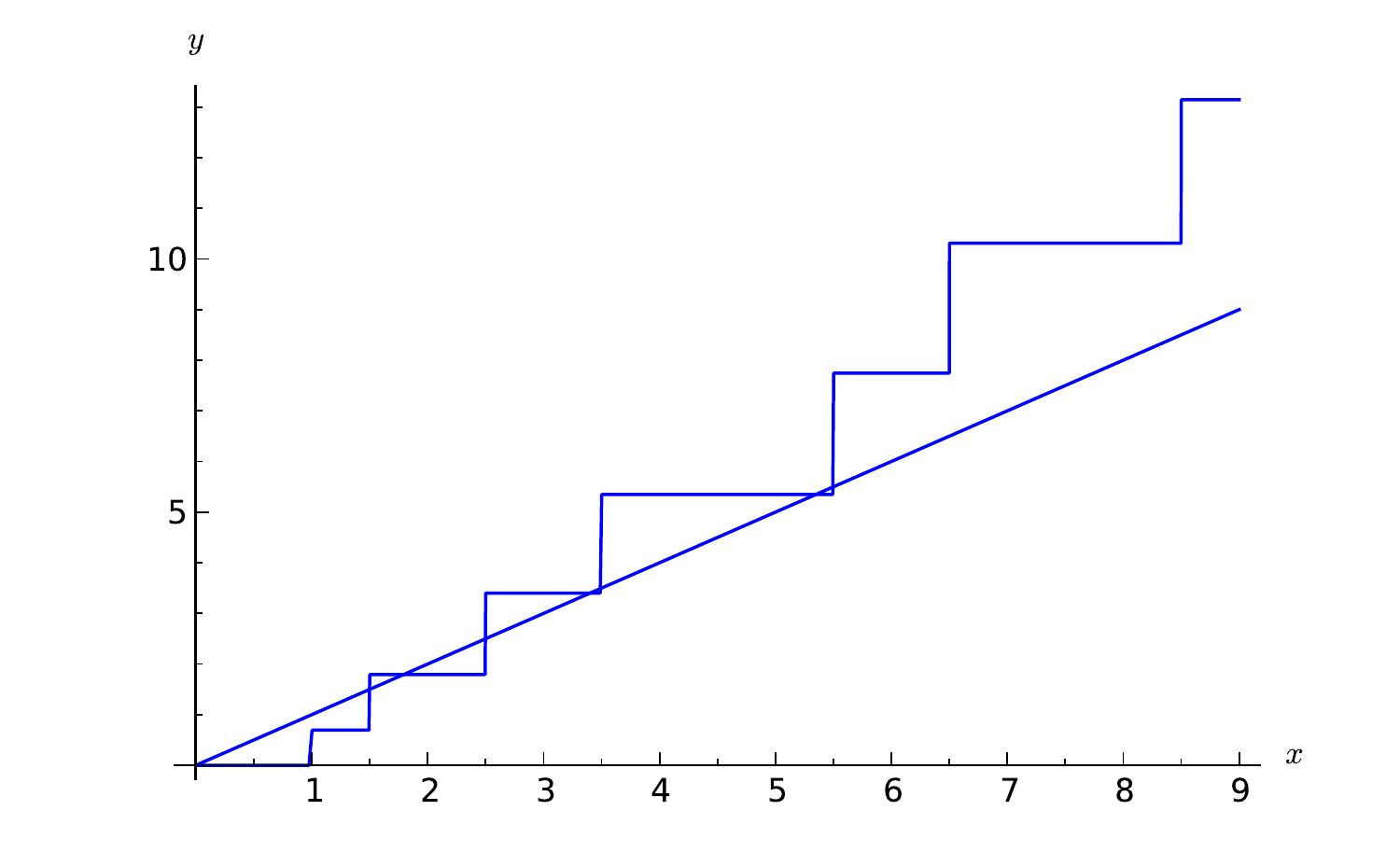}
    \caption{Graphs of $y=\theta(2x)$ and $y=x$}
  \end{figure*}
\end{proof}

\section{Asymptotic upper bounds}

The upper bound appearing in the statement of Theorem~\ref{thm:main} was
chosen because it is a very simple function and because it holds for all
values of the level $N$.  However, the reader will have realised from
the estimates we used that this bound gets less and less sharp as $N$
increases (because $\theta(x)\sim x$ by the Prime Number Theorem).  
We use some known results on the behaviour of the prime
gaps $g_k=p_{k+1}-p_k$ to give better unconditional and conditional 
asymptotic upper bounds.

\begin{thm}
  Let $f$ and $g$ be cuspidal eigenforms of weights $k_1\neq k_2$ on the 
  group $\Gamma_0(N)$.  Then
  \begin{enumerate}
    \item there exists 
      \begin{equation*}
        n=O\left(\left( \log(N)+\log(N)^{0.525}\right)^2\right)
      \end{equation*}
      such that $a_n(f)\neq a_n(g)$;
    \item assuming the Riemann hypothesis, there exists 
      \begin{equation*}
        n=O\left(\left( \log(N)+\log(N)^{0.5}\log\log(N)\right)^2\right)
      \end{equation*}
      such that $a_n(f)\neq a_n(g)$;
    \item assuming Cram\'er's conjecture on prime gaps 
      (see~\cite{Granville}), there exists 
      \begin{equation*}
        n=O\left(\left( \log(N)+(\log\log(N))^2\right)^2\right)
      \end{equation*}
      such that $a_n(f)\neq a_n(g)$.
  \end{enumerate}
\end{thm}
\begin{proof}
  The key point is to estimate the size of the smallest prime $p(N)$ not
  dividing $N$ in terms of $N$.  Since we are looking for an upper bound,
  we are naturally led to focus on the worst-case scenario,
  the \emph{primorials}
  \begin{equation*}
    N_k=p_1p_2\ldots p_k=\exp(\theta(p_k)),
  \end{equation*}
  for which we clearly have $p(N_k)=p_{k+1}$.  Writing $g_k=p_{k+1}-p_k$
  for the gap between consecutive primes, we have
  \begin{equation*}
    p_{k+1}=p_k+g_k=\begin{cases}
      p_k+O(p_k^{0.525}) & \text{unconditional, see Baker-Harman-Pintz~\cite{Baker-Harman-Pintz}}\\
      p_k+O(\sqrt{p_k}\log(p_k)) & \text{assuming RH, see Cram\'er~\cite{Cramer}}\\
      p_k+O(\log^2(p_k)) & \text{assuming Cram\'er's conjecture, see~\cite{Granville}}.
    \end{cases}
  \end{equation*}
  Simple manipulations together with the fact that 
  $p_k\sim\theta(p_k)=\log(N_k)$ give us the upper bounds in the statement.
\end{proof}

\section{A numerical experiment}
Since the results in the previous section all build upon Murty's approach
in Lemma~\ref{lem:murty}, it is natural to ask how tight the bound of
Lemma~\ref{lem:murty} is.

Let us consider the level $1$ case.  Here, Lemma~\ref{lem:murty} tells us
that there exists $n\leq 4$ such that $a_n(f)\neq a_n(g)$.  Is it possible
to improve on the $4$?  We investigated this question via a computational
approach, which we will describe after we recall the following

\begin{conj}[Maeda, Conj. 1.2 in~\cite{Hida-Maeda}]\label{conj:maeda}
  The characteristic polynomial of the Hecke operator $T_2$ acting on
  the space of cusp forms $S_k(\SL_2(\ZZ))$ is irreducible.
\end{conj}

Maeda's conjecture has been verified numerically by 
Farmer-James~\cite{Farmer-James}, Buzzard~\cite{Buzzard}, 
Stein, and 
Kleinerman\footnote{See~\url{http://wstein.org/Tables/charpoly_level1/t2/}}
for all weights $\leq 3000$, except for 
$2796$.  We have verified the case $k=2796$ using the mathematical software
Sage~\cite{Sage} (the computation of the characteristic polynomial used
native Sage code, and the check for irreducibility used polynomial
factorisation code from PARI/GP~\cite{Pari}).

Based on a sample of results for small weights, our project was to compute
the Fourier coefficient $a_2$ of all cuspidal eigenforms of level $1$ and
weights $\leq 10000$.  To each $a_2$ we associate its characteristic
polynomial over $\QQ$.  If we assume Maeda's conjecture, then if we want
to detect duplicates in the list of $a_2$'s, it suffices to look for
duplicates in the list of characteristic polynomials.  To make our
search even more efficient, instead of computing (and storing) the
characteristic polynomial corresponding to each weight $k$, we simply
compute and store the degree and the trace of $T_2$ on the space
$S_k(\SL_2(\ZZ))$.  This reduces the computations that we need to perform
to the following:
\begin{enumerate}
  \item Find the Victor Miller basis of $S_k(\SL_2(\ZZ))$
    \begin{eqnarray*}
      q\phantom{+q^2+\ddots+q^d} &+& \ldots=:f_1(q)\\
      q^2\phantom{+\ddots+q^d} &+& \ldots=:f_2(q)\\
      \ddots\phantom{+q^d} &\vdots& \\
      q^d &+& \ldots=:f_d(q),
    \end{eqnarray*}
    where $d=d_k=\dim S_k(\SL_2(\ZZ))$ and each $q$-expansion is
    computed up to and including the coefficient of $q^{2d}$ (the precision
    required for computing the action of $T_2$ in the next step).
  \item For $j=1,\ldots,d$, compute the coefficient of $q^j$ in $T_2 f_j$,
    given by $a_{2j}(f_j)+2^{k-1}a_{j/2}(f_j)$.
  \item The trace $t_k$ of $T_2$ is the sum of the coefficients computed
    in the previous step.  Store the pair $(d_k, t_k)$. 
\end{enumerate}
This algorithm was implemented in Sage~\cite{Sage} and run in parallel
(one instance per value of $k$) on Linux servers 
\{cerelia, skadi, soleil\}.ms.unimelb.edu.au at the University of Melbourne,
and \{geom, mod, sage\}.math.washington.edu at the University of 
Washington\footnote{We thank William Stein for giving us access to these
servers, which were obtained with support from the US National Science 
Foundation under Grant No. DMS-0821725.}.

After running\footnote{The instance $k=10000$ required about 50 minutes on one
core of a Quad-Core AMD Opteron 8356 processor, and 5.4Gb of memory.} 
this algorithm over the range of weights $2\leq k\leq 10000$, we
found that the list of pairs $(d_k, t_k)$ contained no duplicates.  We
record this result as

\begin{thm}\label{thm:level1}
  Comparing the Fourier coefficient $a_2$ is sufficient to distinguish
  all cuspidal eigenforms of level $1$ and weights $\leq 3000$.  If we
  assume Maeda's conjecture, the same is true for weights $\leq 10000$.
\end{thm}

\bibliographystyle{amsplain}
\bibliography{newforms}
\end{document}